\documentclass[10pt]{amsart}
\usepackage{amssymb}
\usepackage{epsfig}
\usepackage{url}
\usepackage{setspace}
\usepackage{pdflscape}
\theoremstyle{plain}

\usepackage{enumerate} 
\usepackage{tikz} 
\usetikzlibrary{arrows}
\usepackage{caption} 

\newtheorem{thm}{Theorem}[section]
\newtheorem{cor}[thm]{Corollary}
\newtheorem{lem}[thm]{Lemma}

\newtheorem{rem}[thm]{Remark}

\newtheorem{conj}[thm]{Conjecture}

\newtheorem{prob}[thm]{Problem}
\def\cal{\mathcal}
\def\bbb{\mathbb}
\def\op{\operatorname}
\renewcommand{\phi}{\varphi}

\newcommand{\N}{\bbb{N}}
\newcommand{\Z}{\bbb{Z}}
\newcommand{\Q}{\bbb{Q}}

\newcommand{\C}{\bbb{C}}

\begin{document}

\title[Formal inverse of the Prouhet-Thue-Morse sequence]{On formal inverse of the Prouhet-Thue-Morse sequence}
\author{Maciej Gawron and Maciej Ulas}

\thanks{The research of the authors was supported by the grant of the Polish National Science Centre no. UMO-2012/07/E/ST1/00185}

\begin{abstract}
Let $p$ be a prime number and consider a $p$-automatic sequence ${\bf u}=(u_{n})_{n\in\N}$ and its generating function $U(X)=\sum_{n=0}^{\infty}u_{n}X^{n}\in\mathbb{F}_{p}[[X]]$. Moreover, let us suppose that $u_{0}=0$ and $u_{1}\neq 0$ and consider the formal power series $V\in\mathbb{F}_{p}[[X]]$ which is a compositional inverse of $U(X)$, i.e., $U(V(X))=V(U(X))=X$. In this note we initiate the study of arithmetic properties of the sequence of coefficients of the power series $V(X)$. We are mainly interested in the case when $u_{n}=t_{n}$, where $t_{n}=s_{2}(n)\pmod{2}$ and ${\bf t}=(t_{n})_{n\in\N}$ is the Prouhet-Thue-Morse sequence defined on the two letter alphabet $\{0,1\}$. More precisely, we study the sequence ${\bf c}=(c_{n})_{n\in\N}$ which is the sequence of coefficients of the compositional inverse of the generating function of the sequence ${\bf t}$. This sequence is clearly 2-automatic. We describe the sequence ${\bf a}$ characterizing solutions of the equation $c_{n}=1$. In particular, we prove that the sequence ${\bf a}$ is 2-regular. We also prove that an increasing sequence characterizing solutions of the equation $c_{n}=0$ is not $k$-regular for any $k$. Moreover, we present a result concerning some density properties of a sequence related to ${\bf a}$.

\end{abstract}

\maketitle

\section{Introduction}\label{sec1}
Let $k\in\N_{\geq 2}$ and consider a $k$-automatic sequence ${\bf u}=(u_{n})_{n\in\N}$.  Let us recall that the sequence ${\bf u}$ is $k$-automatic if its $n$-th term is generated by a finite automaton with $n$ in base $k$ as the input. One can prove that this property is equivalent to the fact that the family of sequences (called $k$-kernel of ${\bf u}$)
$$
\cal{K}({\bf u}):=\{(u_{k^{a}n+b})_{n\in\N}:\;a\in\N,0\leq b\leq k^{a}-1\}
$$
is finite. The simplest $k$-automatic sequences are periodic sequences.

A famous 2-automatic sequence which is not periodic is the Prouhet-Thue-Morse sequence ${\bf t}=(t_{n})_{n\in\N}$ (PTM sequence for short). In order to define the sequence ${\bf t}$ (on alphabet $\{0,1\}$) let $n\in\N$ be written in base 2, i.e., $n=\sum_{i=0}^{k}\epsilon_{i}2^{i}$, where $\epsilon_{i}\in\{0,1\}$ for $i=0,1,\ldots,k$. Then we define the sum of digits function $s_{2}:\;\N\rightarrow \N$ as $s_{2}(n)=\sum_{i=0}^{k}\epsilon_{i}$. This function satisfies the obvious recurrence relations:
$$
s_{2}(0)=0, \quad s_{2}(2n)=s_{2}(n), \quad s_{2}(2n+1)=s_{2}(n)+1
$$
for $n\geq 0$. The sum of digits function allows us to define the PTM sequence ${\bf t}={(t_{n})}_{n\in\N}$, where $t_{n}=s_{2}(n)\pmod{2}$. We thus have
$$
t_{0}=0, \quad t_{2n}=t_{n}, \quad t_{2n+1}=1-t_{n}
$$
for $n\geq 0$. In particular, from the above relations we immediately deduce that the PTM sequence is indeed 2-automatic. This is clear due to the fact that its kernel contains exactly two sequences, i.e., $\cal{K}({\bf t})=\{{\bf t}, 1-{\bf t}\}$. 

The PTM sequence has many interesting properties and applications in combinatorics, algebra, number theory, topology and even physics (see for example \cite{AllSh1} and \cite{AllSh2}).

Now, if $k=p$ is a prime number then Christol theorem says that the sequence ${\bf u}$ (with terms in a finite field $\mathbb{F}_{p}$) is $p$-automatic if and only if the formal power series $U(X)=\sum_{n=0}^{\infty}u_{n}X^{n}$ is algebraic over $\mathbb{F}_{p}(X)$. In this context one can ask the following general

\begin{prob}\label{genprob}
Let ${\bf u}=(u_{n})_{n\in\N}$ be a $p$-automatic sequence and let $U(X)=\sum_{n=0}^{\infty}u_{n}X^{n}\in\mathbb{F}_{p}[[X]]$ be a formal power series related to the sequence ${\bf u}$. Let us suppose that there exists a formal power series $V(X)=\sum_{n=0}^{\infty}v_{n}X^{n}\in\mathbb{F}_{p}[[X]]$ which is compositional inverse to the series $U$, i.e., $U(V(X))=X$. What can be said about properties of the sequence ${\bf v}=(v_{n})_{n\in\N}$?
\end{prob}

It is well known that the formal power series $U(X)=\sum_{n=0}^{\infty}u_{n}X^{n}$ with coefficients in a commutative ring $R$ is invertible (in the sense of composition) if and only if $u_{0}=0$ and $u_{1}$ is an invertible element of $R$. Moreover, if $U(V(X))=X$ then $V(U(V(X)))=V(X)$ and thus $H(V(X))=0$, where $H(X)=V(U(X))-X$. Because $U$ is non-constant so is $V$ and thus $H\equiv 0$ which is equivalent with the equality $V(U(X))=X$. In particular, the above problem does not make sense for all $p$-automatic sequences. However, we observe that if the problem above is correctly stated then the sequence ${\bf v}$ is $p$-automatic which is an immediate consequence of Christol theorem \cite[Theorem 12.2.5]{AllSh2}. Indeed, if $H\in\mathbb{F}_{p}(X)[Y]$ is a non-zero polynomial with the root $U(X)$, i.e., $H(X,U(X))=0$, then we clearly have $H(V(X),X)=0$ and thus we get $p$-automaticity of the sequence of coefficients of $V$.

As we were unable to prove any general result for Problem \ref{genprob} we concentrate on this problem in case of ${\bf u}={\bf t}$, where ${\bf t}$ is the PTM sequence. To be more precise, let us consider the increasing sequence $(o_{n})_{n\in\N}$
satisfying the equality
$$
\cal{O}:=\{m:\;t_{m}=1\}=\{o_{n}:\; n\in\N_{+}\},
$$
i.e., $o_{n}$ is the $n$-th element of the set $\cal{O}$ of so called ``odious'' numbers. A positive integer is an odious number if the number of 1's in its binary expansion is odd. This is equivalent to the identity $t_{n}=s_{2}(n)\pmod{2}=1$. It is interesting that the sequence $(o_{n})_{n\in\N}$ satisfies $o_{1}=1, o_{2}=2, o_{3}=4$ and for $n\geq 1$ we have the following recurrence relations
$$
\begin{array}{lll}
o_{4n}    & = & o_{n}-3o_{n+1}+3o_{2n+1} \\
o_{4n+1}  & = & -2o_{n+1}+3o_{2n+1} \\
o_{4n+2}  & = & -o_{n}-9o_{n+1}-o_{2n}+8o_{2n+1}\\
o_{4n+3}  & = &-\frac{5}{3}o_{n}-11o_{n+1}-\frac{5}{3}o_{2n}+10o_{2n+1}.
\end{array}
$$
(One can prove that in fact $o_{n}=2n-1-t_{n-1}$ for $n\geq 1$). In particular, we deduce that the sequence $(o_{n})_{n\in\N}$ is 2-{\it regular} \cite{AllSha0, AllSha}. The concept of a $k$-regular sequence is a generalization of $k$-automatic sequences to the case of infinite alphabets. More precisely, we say that the sequence ${\bf u}=(u_{n})_{n\in\N}$ taking values in a $\Z$-module $R$ is a $k$-regular sequence if there exist a finite number of sequences over $R$, say $\{(r_{j,n})_{n\in\N}:\;j=1,2,\ldots, m\}$ such that for each integer $i\in\N$ and $b\in\{0,1,\ldots, k^{i}-1\}$ we have
$$
u_{k^{i}n+b}=\sum_{j=1}^{m}b_{j}r_{j,n}
$$
for some $b_{1},\ldots, b_{m}\in\Z$ and each $n\in\N$. In other words, the $\Z$-module $R$ is finitely generated. Let us also note that the integer sequence, say $(e_{n})_{n\in\N}$, enumerating the set
$$
\cal{E}:=\{m\in\N:\;t_{m}=0\}
$$
of ``evil'' numbers, satisfies the same recurrence relation as the sequence $(o_{n})_{n\in\N}$ (with different initial conditions of course). In particular this sequence is 2-regular too. Let us note that several examples of $q$-automatic sequences which is characteristic function of an $k$-regular increasing sequence of integers are given in \cite[p. 99-105]{Cat}.

2-regularity of the sequences related to the sets $\cal{O}$ and $\cal{E}$ is interesting. Indeed, this is a strong property due to the fact that for a general 2-automatic sequence ${\bf u}$ the sequence related to the set $\{m:\;u_{m}=1\}$ need not be $k$-regular for any $k$. Indeed, let us consider the characteristic sequence of powers of 2, i.e., the sequence ${\bf u}=(u_{n})_{n\in\N_{+}}$ satisfying $u_{1}=1$ and
$$
u_{2n}=u_{n}, \quad u_{2n+1}=0
$$
for $n\geq 1$. We then have an obvious equality $\{m:\;u_{m}=1\}=\{2^{n}:\;n\in\N\}$. It is also clear that the sequence $(2^{n})_{n\in \N}$ is not $k$-regular due to the fact that $k$-regular sequences grow polynomially fast (see \cite[Theorem 16.3.1]{AllSh2}).

Because $t_{0}=0, t_{1}=1$ we note that for the formal power series
$$
F(X)=X+X^2+X^4+X^6+X^7+\ldots=\sum_{n=1}^{\infty}t_{n}X^{n}\in\mathbb{F}_{2}[[X]]
$$
there is a formal power series $G\in\mathbb{F}_{2}[[X]]$ such that $F(G(X))=G(F(X))=X$, i.e., $G$ is a compositional inverse of $F$. We write
$$
G(X)=X+X^2+X^7+X^8+X^9+X^{10}+\ldots=\sum_{n=1}^{\infty}c_{n}X^{n}\in\mathbb{F}_{2}[[X]]
$$
and call the sequence ${\bf c}=(c_{n})_{n\in\N}$  {\it the inverse Prouhet-Thue-Morse sequence} (iPTM for short).

Let us describe the content of the paper in some details. In Section \ref{sec2} we get 2-automaticity of ${\bf c}$ and present some recurrence relations satisfied by this sequence. We also prove transcendence over $\C(X)$ of the power series $C(X)=\sum_{n=0}^{\infty}c_{n}X^{n}$, where we treat $C$ as an element of the ring of formal power series $\C[[X]]$. In Section \ref{sec3} we prove 2-regularity of an increasing sequence ${\bf a}=(a_{n})_{n\in\N}$ with the property
$$
c_{m}=1\Longleftrightarrow m=a_{n}\quad\mbox{for some}\quad n\in\N.
$$
Moreover, we prove that the set of fractions $\{\frac{a_{n}}{n^2}:\;n\in\N_{+}\}$ is dense (in the Euclidean topology) in the set $\left[\frac{1}{6},\frac{1}{2}\right]$. Section \ref{sec4} is devoted to the study of an increasing sequence ${\bf d}=(d_{n})_{n\in\N}$ satisfying the condition
$$
c_{m}=0\Longleftrightarrow m=d_{n}\quad\mbox{for some}\quad n\in\N.
$$
We prove that the sequence ${\bf d}$ is not $k$-regular for any $k\in\N_{\geq 2}$.

 Finally, in the last section we offer some problems and conjectures which may stimulate further research.

\section{Recurrence relation for the sequence ${\bf c}$ and transcendence of a related series}\label{sec2}

In this section we are interested in the coefficients of the series $G(X)=\sum_{n=1}^{\infty}c_{n}X^{n}$ which is the compositional inverse of the Prouhet-Thue-Morse generating series $F(X)$, i.e., we are interested in $G(X)$ which satisfies
\begin{equation*}
F(G(X))=G(F(X))=X
\end{equation*}
over $\mathbb{F}_{2}[[X]]$. Before we state our first result let us recall that the Prouhet-Thue-Morse power series $F(X)$ satisfies the algebraic equation \cite[p. 352]{AllSh2}
\begin{equation}\label{PTMeq}
(1+X)^{3}F(X)^2+(1+X^2)F(X)+X=0.
\end{equation}

We have the following result:

\begin{thm}
The series $G(X)=\sum_{n=1}^{\infty}c_{n}X^{n}$ satisfies each of the following polynomial equations
\begin{align*}
&X^2G(X)^3+X(1+X)G(X)^2+(X^2+1)G(X)+X(X+1)=0,\\
&X^3G(X)^4+(1+X)G(X)+X(X^2+1)=0.
\end{align*}
In particular, the sequence ${\bf c}=(c_{n})_{n\in\N}$ satisfies $c_{0}=0, c_{1}=c_{2}=1, c_{3}=0$ and for $n\geq 1$ we have the following recurrence relations
\begin{equation*}
\begin{array}{lcl}
  c_{4n}   & = & c_{4n-1} \\
  c_{4n+1} & = & c_{4n-1} \\
  c_{4n+2} & = & c_{4n-1} \\
  c_{4n+3} & = & (c_{4n-1}+c_{n})\pmod{2}.
\end{array}
\end{equation*}
\end{thm}
\begin{proof}
We are working in the ring $\mathbb{F}_{2}[[X]]$. Rewriting the equation (\ref{PTMeq}) in terms of $X$ we have $F(X)^2X^3+F(X)(F(X)+1)X^2+(F(X)^2+1)X+F(X)(F(X)+1)=0$. We replace now $X$ by $G(X)$, where the power series $G$ is a formal inverse of $F$, i. e. satisfies $F(G(X))=G(F(X))=X$, and get the equation
\begin{equation}\label{firsteq}
X^2G(X)^3+X(1+X)G(X)^2+(X^2+1)G(X)+X(X+1)=0
\end{equation}
which is exactly the first equation from the statement of our theorem.

In order to get the second equation we multiply (\ref{firsteq}) by $G(X)$ and in the resulting equation we replace $G(X)^3$ by the expression $\frac{(1+X)}{X}G(X)^2+\frac{(X^2+1)}{X^2}G(X)+\frac{X+1}{X}$ (which is just solution of (\ref{firsteq}) with respect to $G(X)^3$). Performing now all necessary simplifications we get the second equation presented in the statement of our theorem.

Let us write $G(X)=\sum_{n=0}^{\infty}c_{n}X^{n}$. In order to get recurrence relations for $c_{n}$ we observe that the second equation for $G(X)$ can be rewritten as
$$
X^3G(X^4)+(1+X)G(X)+X(X^2+1)=0.
$$
We thus get that
\begin{equation*}
\sum_{n=0}^{\infty}c_{n}X^{4n+3}+(1+X)\sum_{n=0}^{\infty}c_{n}X^{n}+X^3+X=0
\end{equation*}
and a quick inspection of coefficients of the above power series gives the equalities $c_{0}=0, c_{1}=c_{2}=1$. Moreover, we have $c_{n}+c_{4n+3}+c_{4n+2}=0$ and $c_{4n+i}+c_{4n+i-1}=0$ for $i=0,1,2$ (here the equalities are taken $\pmod{2}$). These relations are clearly equivalent to those presented in the statement of our theorem.
\end{proof}

Based on the recurrence relations for ${\bf c}$ we can easily construct a finite automaton which generates the sequence ${\bf c}$. It takes the form:

\begin{center}
  \begin{tikzpicture}[->,>=stealth',shorten >=1pt,auto,node distance=2cm,
  thick,main node/.style={circle,draw,font=\sffamily\bfseries},
  in node/.style={circle,fill=gray!50,draw,font=\sffamily\bfseries}]

  \node[in node] (1) {0};
  \node[main node] (3) [below of=1] {1};
  \node[main node] (2) [left of=3] {0};
  \node[main node] (4) [right of=3] {0};
  \node[main node] (5) [below of=4] {0};

  \path[every node/.style={font=\sffamily\small}]
    (1) edge [right] node[above] {3} (2)
        edge [right] node[above] {1,2} (3)
        edge [right] node[above] {0} (4)
    (2) edge [loop left] node[below] {3} (2)
        edge [right] node[above] {1} (3)
        edge [bend right] node[below] {2,0} (5)
    (3) edge [loop below] node[below] {2,0} (3)
        edge [right] node[above] {1,3} (5)
    (4) edge [loop right] node[below] {0} (4)
        edge [right] node[above] {1,3} (5)
        edge [right] node[above] {2} (3)
    (5) edge [loop below] node[below] {0,1,2,3} (5);

\end{tikzpicture}

\end{center}
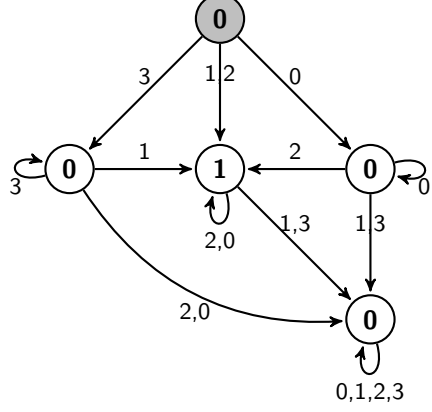
\begin{center}Figure 1. An automaton generating the $iPTM$ sequence \end{center}

\bigskip

In the sequel we will need different recurrence relations for the sequence ${\bf c}$. More precisely, we have:

\begin{lem}\label{alternative}
The sequence ${\bf c}=(c_{n})_{n\in\N_{+}}$ satisfies $c_{0}=0, c_{1}=c_{2}=1, c_{3}=0$. Moreover, we have the following recurrence relations
\begin{equation}\label{charofc}
\begin{array}{lccccccclc}
  c_{8n-1} & = & c_{8n}   & = & c_{8n+1} & = & c_{8n+2} & = & c_{2n-1}  &\text{for}\quad n\geq 1,\\
  c_{8n+3} & = & c_{8n+4} & = & c_{8n+5} & = & c_{8n+6} & = & 0         &\text{for}\quad n\geq 0. \\
\end{array}
\end{equation}
\end{lem}
\begin{proof}
In the proof all equalities are taken $\pmod{2}$. We start with the simple property of the sequence ${\bf c}$ which states that $c_{2n-1}+c_{2n}\equiv 0\pmod{2}$. Indeed, if $n=2m$ then using the recurrence relation, we have
\begin{equation*}
c_{2n-1}+c_{2n}=c_{4m-1}+c_{4m}=2c_{4m-1}=0.
\end{equation*}
In the case $n=2m+1$ we have
\begin{equation*}
c_{2n-1}+c_{2n}=c_{4m+1}+c_{4m+2}=2c_{4m-1}=0.
\end{equation*}

We are ready to prove the result. Let $i\in\{0,1,2\}$. We then have $c_{8n+i}=c_{4(2n)+i}=c_{8n-1}$. In order to prove the equality $c_{8n-1}=c_{2n-1}$ we use induction on $n$. We have $c_{7}=c_{1}=1$. Let us suppose that $c_{8n-1}=c_{2n-1}$. We then have
\begin{align*}
c_{8(n+1)-1}&=c_{4(2n+1)+3}=c_{4(2n)+3}+c_{2n+1}=c_{4(2n)-1}+c_{2n}+c_{2n+1}\\
            &=c_{8n-1}+c_{2n}+c_{2n+1}=c_{2n-1}+c_{2n}+c_{2n+1}=c_{2n+1},
\end{align*}
where in the two last equalities we used induction hypothesis and the remark from the beginning of the proof respectively.

In order to prove the second part of relations for the sequence ${\bf c}$ we note that for $i\in\{0,1,2\}$ we have $c_{8n+4+i}=c_{4(2n+1)+i}=c_{8n+3}$. We proceed by induction on $n$. We have $c_{3}=0$. Let us suppose that $c_{8n+3}=0$ for some $n>0$. We then have
\begin{align*}
c_{8(n+1)+3}&=c_{4(2n+2)+3}=c_{4(2n)+7}+c_{2n+2}\\
            &=c_{2(n+1)-1}+c_{2n+2}=c_{2n+1}+c_{2n+2}=0.
\end{align*}
In the third equality we used the identity $c_{8n+7}=c_{8(n+1)-1}=c_{2n+1}$. The last equality follows from the remark given on the beginning of the proof.
\end{proof}

\begin{rem}{\rm  The sequence ${\bf c}$ take only the values 0 and 1 and thus can be considered as a sequence of real (complex) numbers. In this case we can easily write the functional equation satisfied by the formal power series
\begin{equation*}
C(X)=\sum_{n=0}^{\infty}c_{n}X^{n}\in\C[[X]]
\end{equation*}
and even prove that the series $C$ is transcendental over $\C(X)$.

More precisely, we have the following

\begin{thm}\label{functequation}
The formal power series $C(X)\in\C[[X]]$ satisfies the following functional equation
\begin{equation}\label{functionalequation}
C(X)=X(X+1)+\frac{X^3(X^4-1)}{(X-1)(X^4+1)}C(X^4).
\end{equation}
Moreover, the function $C(X)$ is transcendental over $\C(X)$.
\end{thm}
\begin{proof}
Based on Lemma \ref{alternative} we notice the following chain of equalities
\begin{align}\label{first}
C(X)&=\sum_{n=0}^{\infty}c_{n}X^{n}\\ \notag
    &=X+X^2+\sum_{i=-1}^{2}\sum_{n=1}^{\infty}c_{8n+i}X^{8n+i}\\ \notag
    &=X+X^2+\sum_{i=-1}^{2}\sum_{n=1}^{\infty}c_{2n-1}X^{8n+i}\\ \notag        &=X+X^2+\sum_{n=1}^{\infty}c_{2n-1}\left(\sum_{i=-1}^{2}X^{i+1}\right)X^{8n-1}\\ \notag
    &=X+X^2+\frac{X^4-1}{X-1}\sum_{n=1}^{\infty}c_{2n-1}X^{8n-1}\\ \notag
    &=X+X^2+\frac{1}{2}X^3\frac{X^4-1}{X-1}(C(X^4)-C(-X^4)),\notag
\end{align}
where in the last equality we used the simple equality
$$
\sum_{n=1}^{\infty}c_{2n-1}X^{8n-1}=\frac{1}{2}X^3(C(X^4)-C(-X^4)).
$$
Replacing $X$ by $-X$ in the above equation and adding these two equations we get after simple manipulations the equation
\begin{equation}\label{second}
C(-X)=\frac{X-1}{X+1}C(X).
\end{equation}
Replacing $X$ by $X^4$ in (\ref{second}) we get expression for $C(-X^4)$ in terms of $C(X^4)$. Putting obtained expression into the equation (\ref{first}) and performing necessary simplifications we get the functional equation from the statement of our theorem.

The transcendence of $C(X)$ will be a consequence of a well known result of Nishioka \cite{Nis}. This result says that the power series with rational coefficients which defines a holomorphic function, say $f$, in some neighborhood of zero  and satisfying functional equation of the form $f(X)=r_{1}(X)+r_{2}(X)f(X^m)$ for some $m\in\N_{\geq 2}$ and functions $r_{1}, r_{2}\in\C(X)$, is either rational or transcendental over $\C(X)$. It is clear that the power series $C(X)$ defines a holomorphic function in the domain $|X|<1$.

Let us assume that the function $C(X)$ is rational. This means that there are co-prime polynomials $p, q\in\C[X]$ such that $C(X)=p(X)/q(X)$. Moreover, we can assume that $q(X)\neq 0$. Putting this expression into the functional equation (\ref{functionalequation}) and multiplying both sides by $(X-1)(X^4+1)q(X)q(X^4)$ we get
\begin{equation*}
(X-1)(X^4+1)p(X)q(X^4)=X(X^2-1)(X^4+1)q(X)q(X^4)+X^3(X^4-1)p(X^4)q(X)
\end{equation*}
or equivalently
\begin{equation}\label{equation1}
(X-1)(X^4+1)[p(X)-X(X+1)q(X)]q(X^4)=X^3(X^4-1)p(X^4)q(X).
\end{equation}
In particular $q(X^4)|(X^4-1)p(X^4)q(X)$ due to the inequality $q(0)\neq 0$. Moreover, because the polynomials $p(X), q(X)$ are co-prime then $p(X^4)$ and $q(X^4)$ are co-prime too. We thus have $q(X^4)|(X^4-1)q(X)$. We thus have that $4\op{deg}q\leq 4+\op{deg}q$ and thus $\op{deg}q\leq 4/3<2$ which implies that $\op{deg}q\in\{0,1\}$. However, if $\op{deg}q=0$ then comparing the degrees of both sides in (\ref{equation1}) we get
$$
7+4\op{deg}p=5+\op{deg}(p(X)-X(X-1)q(X))\leq 5+\op{max}\{\op{deg}p,2\}.
$$
This inequality is clearly impossible.

Finally, if $\op{deg}q=1$ then using similar reasoning we get the equality
$$
5+\op{deg}(p(X)-X(X-1)q(X))=8+4\op{deg}p
$$
or equivalently $\op{deg}(p(X)-X(X-1)q(X))=3+4\op{deg}p$. If $\op{deg}p\geq 1$ we get a contradiction. Thus, the polynomial $p$ need to be of degree 0. Let us put $q(X)=aX+b$ and $p(X)=c$. Because $C(X)$ is non-constant we have $ac\neq 0$. However, a quick computation reveals that the leading term of the difference of both sides of the equation (\ref{equation1}) is $ac$ and need to be zero and we get a contradiction.

\end{proof}}
\end{rem}

\begin{rem}
{\rm In the above proof instead of Nishioka result we could use the classical result of Fatou: {\it if a power series $\sum_{n=0}^{\infty}a_{n}z^{n}$ with integer coefficients converges inside the unit disk, then it is either rational or transcendental over $\Q(X)$} \cite{Fat}.}

 \end{rem}

It is clear that if we make a reduction $\pmod{2}$ then the functional equation for $C$ reduces to the equation satisfied by $G$ over $\mathbb{F}_{2}(X)$. Let us also note that using the functional equation for the series $C$ one can deduce an alternative recurrence relation for the sequence ${\bf c}$. Indeed, let us write
$$
\frac{X^{3}(X^4-1)}{(X-1)(X^4+1)}=\sum_{n=0}^{\infty}p_{n}X^{n},
$$
i.e.,
$$
p_{n}=\begin{cases}

\begin{array}{lll}
1 &  & \mbox{if}\quad n=0\quad\mbox{or}\quad n\equiv -1,0,1,2\pmod{8} \\
0 &  & \mbox{if}\quad n\equiv 3,4,5,6\pmod{8}
\end{array}
\end{cases}.
$$
Then, comparing the coefficients of both sides of the functional equation (\ref{functionalequation}) we get
$$
c_{n}=\sum_{k=0}^{\lfloor\frac{n}{4}\rfloor}p_{n-4k}c_{k},
$$
for $n\geq 3$.

\section{Characteristic sequence of 1's in the sequence {\bf c}}\label{sec3}

The aim of this section is to give characterization of an increasing sequence ${\bf a}=(a_{n})_{n\in\N}$ satisfying the equality
\begin{equation*}
\cal{A}:=\{m\in\N:\;c_{m}=1\}=\{a_{n}:\;n\in\N\}.
\end{equation*}

We start with characterization of the elements of $\cal{A}$ in terms of their expansions in base 4.

\begin{thm}\label{0characterization}
The set $\cal{A}$ consists of these integers $n > 0$ such that all their $2^{2k}$-th binary digits are 0 for $k\geq 1$, or equivalently expansion of $n+1$ in base $4$ consists only of digits $0$ and $2$, except for the last digit which can be arbitrary.
\end{thm}
\begin{proof}
We will use induction on length of binary expansion of $n+1$. Our theorem holds for $n<7$. Let us suppose that $n+1$ has binary expansion of length $k \geq 3$.

From recurrence relations presented in Lemma \ref{alternative} we get that if $n=8k+r$ with $r\in\{3,4,5,6\}$ then $n \notin \cal{A}$. Let us note that this condition is equivalent to the fact that $n+1$ has digit $1$ at position $2^2$ in its binary expansion.

Next, if $n=8k+r$ with $r\in\{0,1,2,7\}$ then using Lemma \ref{alternative} one more time
we have $c_n = c_{\left[\frac{n-3}{4}\right]}$. It turns out that binary expansion of $\left[\frac{n-3}{4}\right]+1 = \left[\frac{n+1}{4}\right]$ is just binary expansion of $n+1$ with truncated two last digits. In order to finish the proof we use induction hypothesis to get the statement of our theorem.
\end{proof}

In order to enumerate elements of the set $\mathcal{A}$ we define a sequence $(b_n)_{n \in\N}$ of those non-negative integers which has only digits $0$ and $2$ in base $4$ expansion. Number $b_n$ can be computed in the following way: replace each digit $1$ by the digit $2$ in the unique binary expansion of $n$,
 the string obtained in this way is an expansion of $b_n$ in base $4$. Therefore we can write the following recurrence relations
\begin{equation*}
  b_{0}    =  0, \quad b_{2n} = 4b_{n}, \quad b_{2n+1} = 4b_n + 2.
\end{equation*}
Let $(a_n)_{n\in\N}$ be the increasing sequence of elements of $\mathcal{A}$, then we have the following relation between sequences $(a_n)_{n\in\N}$ and $(b_n)_{n\in\N}$.

\begin{lem}\label{abrelation}
For each integer $k\geq 0$ the following equality holds
\begin{equation*}
a_{4k+r} = 4b_k+r  \text{ for } r=-1,0,1,2,
\end{equation*}
provided that $4k+r \geq 0$.
\end{lem}
\begin{proof}
From Theorem \ref{0characterization} we know that the sequence $(a_n)_{n\in\N}$ consists of those positive integers $k$ such that expansion  of $k+1$ in base $4$ has only digits $0$ and $2$, except at the last position, where we can have arbitrary digit.
 Let us look at the following sequence
 \begin{equation*}
4b_0,4b_{0}+1,4b_0+2,4b_{0}+3,  \quad 4b_{1},4b_{1}+1,4b_{1}+2,4b_{1}+3, \quad 4b_2,4b_2+1,4b_2+2,4b_2+3,\;\ldots
 \end{equation*}
It consists of those integers which have only digits $0$ and $2$ in base $4$, except at the last position, where we can have arbitrary digit.
Therefore, to get description of ${\bf a}=(a_n)_{n\in\N}$ we just have to subtract one from each element of this sequence. We want $a_1$ to be the first positive element and thus
\begin{equation*}
a_1 = 4b_{0}+1,\quad a_2 = 4b_0+2,\quad a_3 = 4b_1-1,\quad  a_4 = 4b_1, \quad a_5 = 4b_1+1,\; \ldots .
 \end{equation*}
Clearly, the above relations can be rewritten in the following form
\begin{equation*}
a_{4k+r} = 4b_k+r  \text{ for } r=-1,0,1,2.
\end{equation*}
Our lemma is proved.
\end{proof}

We use the above lemma in order to get 2-regularity of the sequence ${\bf a}$. More precisely, we have the following:

\begin{thm}\label{1characterization}
Let $G(X)=\sum_{n=1}^{\infty}c_{n}X^{n}$ be the inverse power series of the Prouhet-Thue-Morse power series $F(X)$ over $\mathbb{F}_{2}$. Then $c_{m}=1$ if and only if $m=a_{n}$, where $n>0$, $a_{0}=0, a_{1}=1, a_{2}=2, a_{3}=7$ and for $n\geq 1$ we have the following recurrence relations
\begin{equation*}
\begin{array}{lcl}
  a_{4n}   & = & a_{4n-1}+1 \\
  a_{4n+1} & = & a_{4n-1}+2 \\
  a_{4n+2} & = & a_{4n-1}+3 \\
  a_{8n+3} & = & a_{8n}+7   \\
  a_{8n+7} & = & 4a_{4n+3}+3.
\end{array}
\end{equation*}
\end{thm}
\begin{proof}
From Lemma \ref{abrelation} and recurrence relations satisfied by the sequence $(b_n)_{n\in\N}$ we get that
\begin{equation*}
a_{4n+r} = 4b_n+r = (4b_{n}-1)+(r+1) = a_{4n-1}+(r+1) \text{ for } r=0,1,2.
\end{equation*}
Moreover
\begin{equation*}
a_{8n+3} = 4b_{2n+1}-1 = 16b_{n}+7 = 4b_{2n}+7 = a_{8n-1}+8 = a_{8n}+7,
\end{equation*}
and
\begin{equation*}
a_{8n+7} = 4b_{2n+2}-1 = 16b_{n+1}-1 = 4(a_{4n+3}+1)-1 = 4a_{4n+3}+3.
\end{equation*}
So we get desired relations and our theorem is proved.
\end{proof}

From the definition of the sequence $(c_{n})_{n\in\N}$ we easily deduce that it consists of blocks of 0's of length divisible by 4 and blocks of 1's of the exact length 4. Thus, a natural question arises about the behavior of the sequence $(a_{n+1}-a_{n})_{n\in\N}$. In particular, we are interested in its limit points and set of values.

We have the following:

\begin{thm}\label{dif}
We have
\begin{equation*}
 \{a_{n}-a_{n-1}:\;n\in\N_{+}\}=\left\{\frac{1}{3}(4^{m}-1):\;m\in\N\right\}.
\end{equation*}
Moreover,
$$
a_{n}-a_{n-1}=1 \Longleftrightarrow n\equiv 0, 1, 2\pmod{4}, n\in\N_{+},
$$
and for $m\geq 2$ we have
$$
a_{n}-a_{n-1}=\frac{1}{3}(4^{m}-1) \Longleftrightarrow n=2^{m+1}k+2^{m}-1\quad\mbox{for each}\quad k\in\N.
$$
\end{thm}
\begin{proof}
 From Lemma \ref{abrelation} we know that $a_{4k+r} = 4b_k+r$ for $r=-1,0,1,2$, therefore when
 $n \equiv 0,1,2 \pmod 4$ we have $a_n-a_{n-1} = 1$.

 Let us take $n = 4k-1$ and observe that $a_{4k-1}-a_{4k-2} = 4(b_{k}-b_{k-1})-1$.
 Let binary expansion of $k$ be $k = \overline{d_sd_{s-1}\ldots d_0}$, and let $q$ be the least integer such that
 $d_q \neq 0$. The binary expansion of $k-1$ is $k-1 = \overline{d_sd_{s-1}\ldots d_{q+1}011\ldots 1}$. We are interested in computing $b_k-b_{k-1}$. We have
 \begin{equation*}
  b_k-b_{k-1} = \overline{d_s0d_{s-1}0\ldots d_00}-\overline{d_s0d_{s-1}0\ldots d_{q+1}0001010\ldots 10} = 2^{2q+1} - \sum_{i=0}^{q-1} 2^{2i+1}
 \end{equation*}
and as a consequence we get
\begin{equation*}
a_n-a_{n-1} = 4(2^{2q+1} - \sum_{i=0}^{q-1} 2^{2i+1}) -1 = 8\left(4^q - \frac{4^q-1}{3}\right) -1 = \frac{4^{q+2}-1}{3}.
\end{equation*}
We put $m = q-2$ to get $n = 4k-1 = 4(2^q+2^{q+1}l)-1 = 2^{m+1}l+2^m-1$.
Our theorem follows.
\end{proof}

As a immediate consequence of Theorem \ref{dif} we get

\begin{cor}
For each $m\in\N_{+}$ in the sequence ${\bf c}$ there are infinitely many strings of $0$'s of length $\frac{1}{3}(4^{m}-1)$. In particular, the sequence ${\bf c}$ contains an arbitrary long strings of consecutive $0$'s.
\end{cor}

\begin{rem}
{\rm Let us consider the infinite word
$$
{\bf m}=c_{0}c_{1}c_{2}\ldots.
$$
We observe that the above corollary shows that ${\bf m}$ contains arbitrary large powers. This is in strong contrast with the property of the Prouhet-Thue-Morse word ${\bf T}=t_{0}t_{1}t_{2}\ldots$ which does not contain any power of a finite word with exponent $\geq 3$ as a subword.}
\end{rem}

In the next result we present an interesting connection between the sequence ${\bf a}$ and the PTM sequence.

\begin{thm}
Let ${\bf t}=(t_{n})_{n\in\N}$ be the PTM sequence, i.e., $t_{n}=s_{2}(n)\pmod{2}$, and let us consider the sequence ${\bf a}=(a_{n})_{n\in \N}$ defined in Theorem \ref{1characterization}. Then the following identity holds:
\begin{equation*}
t_{a_{n}}=\frac{1}{2}(t_{n}+t_{n+1})+\frac{1}{2}(-1)^{n}(t_{n}-t_{n+1}).
\end{equation*}
In particular $t_{a_{2n}}=t_{n}$ and $t_{a_{2n+1}}=t_{n+1}$ for $n\in\N$.
\end{thm}
\begin{proof}
Let us take an arbitrary integer $k>0$ with binary expansion $k = \overline{d_sd_{s-1}\ldots d_0}$. From definition of the sequence $(b_n)_{n\in\N}$ we get that the binary expansion of $b_k$ takes the form $b_{k}=\overline{d_s0d_{s-1}0\ldots d_00}$. As a consequence we get $s_2(b_k) = s_2(k)$.

Now we are ready to prove our result.
As $a_{4k+r}=4b_{k}+r$ for $r=0,1,2$, we simply compute
\begin{align*}
s_2(a_{4k}) &= s_2(4b_{k}) = s_2(k), \\
s_2(a_{4k+1}) &= s_2(b_k) +1 = s_2(k)+1 = s_2(2k+1), \\
s_2(a_{4k+2}) &= s_2(b_k)+1 = s_2(k)+1 = s_2(2k+1).
\end{align*}

  When $n=4k+3$ then $a_{4k+3}= 4b_{k+1}-1$.
  Let the binary expansion of $k+1$ be $k+1 = \overline{d_sd_{s-1}\ldots d_0}$. We define $m\in\N$ as the minimal integer such that $d_m \neq 0$.
  The binary expansion of $4b_{k+1}$ takes the form $4b_{k+1} = \overline{d_s0d_{s-1}0\ldots d_0 000}$ (it has $2m+3$ ending zeros). When we subtract one from this number we change all ending zeros into ones, and change first one into zero, thus
  $s_2(a_{4k+3}) = s_2(4b_{k+1}-1) = s_2(k+1)-1+(2m+3) \equiv s_2(2k+1) \pmod 2$.
  Our theorem follows.

\end{proof}

\begin{thm}
We have the following equalities
\begin{equation*}
m:=\liminf_{n\rightarrow +\infty}\frac{a_{n}}{n^2}=\frac{1}{6},\quad M:=\limsup_{n\rightarrow +\infty}\frac{a_{n}}{n^2}=\frac{1}{2}.
\end{equation*}
Moreover, the set
$$
\cal{A}:=\left\{\frac{a_{n}}{n^2}:\;n\in\N_{+}\right\}
$$
is dense in $\left[\frac 16,\frac{1}{2}\right]$.
\end{thm}
\begin{proof}
 We know that $a_{4k+r} = 4b_k+r$ for $r=-1,0,1,2$ and thus
 \begin{equation*}
  \liminf_{n \to \infty} \frac{a_n}{n^2} = \liminf_{k \to \infty} \frac{b_k}{4k^2}
  \quad \text{ and } \quad \limsup_{n \to \infty} \frac{a_n}{n^2} = \limsup_{k \to \infty} \frac{b_k}{4k^2}.
 \end{equation*}
We thus see that in order to get the result it is enough to prove that $\displaystyle\limsup_{k \to \infty} \frac{b_k}{k^2} = 2$, $\displaystyle\liminf_{k \to \infty} \frac{b_k}{k^2} = \frac 23$
 and that the set $$
\cal{B}:=\left\{\frac{b_{k}}{k^2}:\;k\in\N_{+}\right\}
$$
is dense in $\left[ \frac 23 , 2 \right]$.

Let us prove that $2n^2 \geq b_n \geq \frac{2}{3}(n^2 +2n)$ for $n \geq 1$. We use induction on $n$. Our inequalities are satisfied for $n=1$.
We consider two cases depending on parity of $n$. If $n=2k$ then we have
\begin{align*}
b_{n}=b_{2k} &= 4b_k \geq 4 \cdot \frac{2}{3}(k^2+2k) \geq \frac{2}{3}(4k^2+4k)=\frac{2}{3}(n^2+2n),\\
b_{n}= 4b_k \leq 8k^2=2n^2.
\end{align*}
In the case when $n=2k+1$ we have
\begin{align*}
b_{n}=b_{2k+1} &= 4b_k+2 \geq 4 \cdot \frac{2}{3}(k^2+2k) +2 \\
               &\geq \frac{2}{3} ( 4k^2+8k+3) \geq \frac{2}{3} ((2k+1)^2 + 2(2k+1))=\frac{2}{3} (n^2 + 2n),\\
b_{n}=b_{2k+1} &= 4b_k+2 \leq 8k^2+2 \leq 8k^2+8k+2 = 2(2k+1)^2=2n^2.
\end{align*}
Our inequalities are proved. As a consequence we obtain the inequalities
$$\frac{2}{3}\leq \displaystyle\liminf_{k \to \infty} \frac{b_k}{k^2}\quad\mbox{and}\quad\displaystyle\limsup_{k \to \infty} \frac{b_k}{k^2} \leq 2.$$
On the other hand $b_{2^k} = 2 \cdot 4^k$ and
$b_{2^k-1} = 2 \cdot \frac{4^{k}-1}{3}$ for each integer $k$ and thus the extremal values of the sequence $(b_{k}/k^2)_{k\in\N_{+}}$ are attained and corresponding values of $m$ and $M$ are as in the statement of the theorem.

Now we prove that the set $\cal{B}$ is dense in $(\frac{2}{3}, 2)$. Let us fix number $q \in (\frac{2}{3}, 2)$ and define the sequence $(u_n)_{n \geq 0}$ as follows: $u_0=1$ and for $n\geq 1$ by
the following relations
\begin{equation*}
 u_n = \begin{cases}
	  2u_{n-1}+1 & \text{ when } \quad q(2u_{n-1}+1)^2 \leq b_{2u_{n-1}+1}, \\
	  2u_{n-1} \quad & \text{ otherwise. }
       \end{cases}
\end{equation*}
We claim that $\displaystyle\lim_{n\to\infty}\frac{b_{u_n}}{u_n^2}=q$. Let us define the set
\begin{equation*}
 E:= \{n \in \N :\;u_n \equiv 0 \pmod 2 \}.
\end{equation*}
We will prove that $E$ is infinite.

First of all we prove that $E$ is non-empty. To the contrary let us suppose that $E = \emptyset$. As a consequence we get that $u_n = 2^{n+1}-1$ and $\displaystyle\lim_{n\to\infty}\frac{b_{u_n}}{u_n^2}  \frac{2}{3} < q$. However, this is a contradiction because $b_{u_n} \geq q u_n^2$ from the definition of $u_n$.

Now, let us suppose that $E$ is non-empty but finite. Thus, one can define the number $n_0 = \max E$. The number $u_{n_0}$ is even and from the definition of our sequence we get $u_{n_0} = 2u_{n_0-1}$
and $b_{2u_{n_0-1}+1} < q(2u_{n_0-1}+1)^2$. The inequality is clearly equivalent to the inequality $b_{u_{n_0}}+2 < q(u_{n_0}+1)^2$.
We have that $u_{n} = 2^{n-n_0} u_{n_0} + 2^{n-n_0}-1$ for
$n \geq n_0$ and thus
\begin{equation*}
\lim_{n\to +\infty} \frac{b_{u_n}}{u_n^2} =\lim_{n\to +\infty} \frac{(4^{n-n_0}b_{u_{n_0}}+2 \cdot \frac{4^{n-n_0}-1}{3})}{(2^{n-n_0} u_{n_0} + 2^{n-n_0}-1)^2} =\frac{(b_{u_{n_0}}+\frac 23)}{(u_{n_0}+1)^2}<q.
\end{equation*}
However, this is a contradiction because for all $n$ we have $b_{u_n} \geq (u_n)^2 q$.

The sequence $\frac{b_{u_n}}{u_n^2}$ is decreasing and thus has a limit $r$. From $b_{u_n}\geq qu_n^2$ we have
$r \geq q$. The set $E$ is infinite and as a consequence we get that for infinitely many $n$ the
inequality $b_{u_{n}}+2 < q(u_{n}+1)^2$ holds. Thus $r = \displaystyle\lim_{n \to \infty} \frac{(b_{u_n}+2)}{(u_{n}+1)^2} \leq q$ and this implies the equality $r=q$. Our theorem follows.

\end{proof}

\section{Characteristic sequence of 0's in the sequence {\bf c}}

Let
$$
\mathcal D: = \{ m\in\N_{+}:\; c_m = 0 \}
$$
and let ${\bf d}=(d_{n})_{n\in\N_{+}}$ be the increasing sequence satisfying the equality $\mathcal{D}=\{d_{n}:\;n\in\N_{+}\}$.
In this section we will prove that the sequence ${\bf d}$ is not $k$-regular for any $k$.
In order to prove this theorem we will need some lemmas concerning auxiliary sequence
$${\bf z}=(z_n)_{n\in\N_{+}},$$
where $z_n = (\frac{d_{4n}+1}4-n) \pmod 2$.

\begin{lem}
 Let us consider the sequence ${\bf z}=(z_{n})_{n\in\N}$.
 Then $z_n = 1$ if and only if $n$ can be written in the form $n = \sum_{i=0}^s 2^{n_i-1}(2^{n_i}-1)$ for some increasing
 sequence $1 < n_0 < \ldots < n_s$.
\end{lem}
\begin{proof}
 We know from Theorem \ref{0characterization} that $m \in \mathcal{D}$ if and only if at least one of the digits in the expansion $m+1 = \sum_{i=0}^t e_i4^i$ of $m+1$ in base~$4$ is equal to $1$ or $3$. Let us put $u_{n}=(d_{4n}+1)/4$. We thus see that the sequence $(u_n)_{n \in\N}$ consists of those positive  integers which have at least one digit equal to $1$ or $3$ in their base-$4$ expansion.

 Let us prove that the elements of the sequence $(u_n)_{n \geq 0}$ satisfy the following recurrence formula:
 \begin{equation*}
  u_0 = 1, u_1 = 3,
  \begin{cases}
    u_{4^n-2^n + i} = 4^n+i & \text{ for } i=0,1,\ldots,4^n-1, n \geq 1, \\
    u_{2\cdot 4^n-2^n +i} = 2 \cdot 4^n +  u_i & \text{ for } i=0,1,\ldots,4^{n}-2^{n}-1, n \geq 1, \\
    u_{3\cdot 4^n-2^{n+1} +i} = 3\cdot 4^n + i & \text{ for } i=0,1,\ldots,4^n-1, n \geq 1. \\
  \end{cases}
 \end{equation*}
In order to prove the above relations we use induction. Moreover, we also prove that $|\{n | u_n < 4^{m} \}| = 4^{m}-2^m$. Let us suppose that our relations holds for $n < m$. The only elements of $(u_n)_{n\in\N}$ which are less than $4$ are $1$ and $3$ and the statement is true for $m=1$.
 Let us consider elements of the sequence $(u_n)_{n\in\N}$ of length $m+1$. When such number starts with digit $1$ or $3$ then the rest of digits can be
 arbitrary. Moreover when the expansion of the number $u_{n}$ starts with digit $2$ then the rest of digits has to form an element of the sequence $(u_n)_{n\in\N}$ of length at most $m$. Therefore
 there are $2 \cdot 4^m + 2(4^m-2^m) = 4^{m+1}-2^{m+1}$ numbers in the sequence $(u_n)_{n\in\N}$ of length at most $m+1$.
 Moreover it is clear from our reasoning that the above recurrence holds.

 Let us prove by induction on $m$ that if $u_n \neq n \pmod 2$ and $u_n \in [4^{m-1}, 4^m)$ then we can write
 $n = \sum_{i=0}^s 2^{n_i-1}(2^{n_i}-1)$ for some increasing sequence $1 < n_0 < \ldots < n_s = m$.
 If $m=1$ then $u_0 = 1 \not\equiv 0 \pmod 2$ and $n=0$ corresponds to the empty sum.
 If $m > 1$ then from our recurrence relation we can see that the only possibility is that
 $n = 2 \cdot 4^{m-1}-2^{m-1} + i$, where $i \leq 4^{m-1}-2^{m-1}-1$ and $u_i \not\equiv i \pmod 2$. From induction hypothesis we get
 $ i = \sum_{j=0}^{s} 2^{n_i-1}(2^{n_i}-1)$ for some increasing sequence $1 < n_0 < \ldots < n_s < m$. Summing up we get
 \begin{equation*}
 n = 2 \cdot 4^{m-1} - 2^{m-1} + i = 2^{m-1}(2^m-1)+i = 2^{m-1}(2^m-1) + \sum_{j=0}^{s} 2^{n_i-1}(2^{n_i}-1),
 \end{equation*}
and thus if we define $n_{s+1} = m$ we arrive at desired equality.
\end{proof}

\begin{lem}\label{lembad}
 Let $m \geq 2$ for each $n \in \left[2^{2m}, \frac 43 2^{2m}\right]$ we have $z_n = 0$.
\end{lem}
\begin{proof}
Let $n\in \left[2^{2m}, \frac 43 2^{2m}\right]$ and suppose that $z_n = 1$. In particular
\begin{equation*}
n  =  \sum_{i=0}^s 2^{n_i-1}(2^{n_i}-1)
\end{equation*}
for some increasing sequence $n_{0},\ldots,n_{s}$. We then have
\begin{equation*}
2^{2m+1}-2^{m+1}=2^{2m+1} \left(1-\frac 1{2^m} \right) > 2^{2m+1} \cdot \frac 23 \geq n \geq 2^{n_s-1}(2^{n_s}-1)
 \end{equation*}
and thus $n_s \leq m$. On the other hand
 \begin{align*}
2^{2m} \leq n = \sum_{i=0}^s 2^{n_i-1}(2^{n_i}-1) &\leq (2^3-2)+\ldots+(2^{2m-1}-2^m)\\
                                                  &<1+2+\ldots+2^{2m-1} < 2^{2m}
 \end{align*}
and we get a contradiction.
\end{proof}

\begin{lem}\label{1divisor}
 For each positive integer $m\geq 2$ there exists a positive integer $n$ such that $m|n$ and $z_n = 1$.
\end{lem}
\begin{proof}
 Let us write $m = 2^u \cdot (2v+1)$. There exists $q = (u+1)\varphi(2v+1)+1$ such that $2^{q-1} \equiv 1 \pmod {2v+1}$ and $2^{q-1} > 2^u$.
 Now we can define
 \begin{equation*}
    n = \sum_{i=1}^{2v+1} 2^{i(q-1)}(2^{i(q-1)+1}-1) \equiv \sum_{i=1}^{2v+1} 1\cdot(2-1) \equiv 0 \pmod {2v+1},
 \end{equation*}
 which satisfies the desired conditions.

\end{proof}

\begin{lem}\label{2bad}
 The sequence ${\bf z}$ is not $2$-automatic.
\end{lem}
\begin{proof}
 Let us suppose that ${\bf z}$ is $2$-automatic. Let $A$ be the automaton with $M$ states which for given $n$ computes $z_n$. We know that $z(2^{M+1}(2^{M+2}-1)) = 1$.
 From pumping lemma \cite[Lemma 4.2.1]{AllSh2} we know that there exists $q>0$ for which we have $z(2^{q+M+1}(2^{M+2}-1)) = 1$. Thus we can write
 \begin{equation*}
  2^{q+M+1}(2^{M+2}-1) = \sum_{i=0}^s 2^{n_i-1}(2^{n_i}-1)\geq 2^{n_0-1}(2^{n_0}-1).
 \end{equation*}
But as $\nu_2(\sum_{i=0}^s 2^{n_i-1}(2^{n_i}-1)) = 2^{n_0-1}$ we have $n_0-1 = q+M+1$ and then
\begin{equation*}
2^{n_0-1}(2^{n_0}-1) > 2^{q+M+1}(2^{M+2}-1)
\end{equation*}
and we get a contradiction.

\end{proof}

\begin{lem}\label{restbad}
 Let $k>1$ be an integer not of the form $2^{m}$ for some positive integer $m$. Then the sequence ${\bf z}$ is not $k$-automatic.
\end{lem}
\begin{proof}
 Let us suppose that ${\bf z}$ is $k$-automatic. Let $A$ be the automaton with $M$ states which for given $n$ computes $z_n$. From Lemma \ref{1divisor} we get
 that there is $n$ which is divisible by $k^{M+1}$ say $n=k^{M+1}L$ such that $z_n = 1$. From pumping lemma \cite{AllSh2} we know that there exists $q>0$ such that
 for all $\nu > 0$ we have $z(k^{M+1+\nu q} L) = 1$. Let us prove that we can find $\nu$ such that
 \begin{equation}\label{badcond}
  k^{M+1+\nu q} L \in \left[2^{2U}, \frac{4}{3} 2^{2U}\right] \text{ for some } U.
 \end{equation}
 Which is equivalent to
 \begin{equation*}
  \{ \log_4(k^{M+1+\nu q} L) \} < \log_4 \frac{4}{3}.
 \end{equation*}
We have
$\{ \log_4(k^{M+1+\nu q} L) \} = \{\nu q \log_4 k + \log_4(k^{M+1}L)\}$ as $\log_4(k^{M+1}L)$ is constant and $\log_4 k \notin \Q$ we know
that $\{ \log_4(k^{M+1+\nu q} L) \}$ is equidistributed in $[0,1]$. So we can find $\nu$ such that condition (\ref{badcond}) holds.
We arrive at a contradiction with Lemma \ref{lembad}.

\end{proof}

Gathering all the proved results together we get the following:

\begin{thm}
The sequence ${\bf d}$ is not $k$-regular for any $k\in\N_{\geq 2}$.
\end{thm}

\begin{proof}
Suppose that the sequence ${\bf d}$ is $k$-regular for some $k$, then the sequence ${\bf z}$ has to be a $k$-automatic sequence for some $k\in\N_{\geq 2}$, which is impossible. However, we get a contradiction with Lemma \ref{2bad} when $k$ is power of 2 and Lemma \ref{restbad} in the opposite case.
\end{proof}

\section{Problems and conjectures}\label{sec4}

In this section we present some problems and conjectures connected with the sequence ${\bf c}$ and related sequences.

Let $p, n\in\N$ and consider the Hankel matrix related to the sequence ${\bf c}$
$$
h_{p,n}=[c_{p+i+j-1}]_{1\leq i, j\leq n}.
$$
Next, let us consider the sequence $H_{p,n}:=\op{det}h_{p,n}$. Recently, there has been a growing interest in studing the arithmetic properties of the sequences of Hankel determinants of automatic sequences (see for example \cite{AllP,Han} and references given there). This motivates us to state the following general

\begin{prob}
Characterize those $p\in\N$ such that the sequence $(H_{p,n})_{n\in\N}$ is $k$-automatic for some $k\in\N_{\geq 2}$.
\end{prob}

Based on numerical computations we believe that the following is true:

\begin{conj}
We have the following properties
\begin{enumerate}
\item For each $p, n\in\N$ we have $H_{p,n}\in\{-1,0,1\}$.

\item We have
\begin{equation*}
H_{0,n}=\begin{cases}\begin{array}{cl}
                       1,  & n=\frac{1}{3}(4^{m}-1)\quad\text{ for some }\quad m\geq 2 ;\\
                       -1, & n\in\{\frac{2}{3}(4^{m}-1)+1, \frac{2}{3}(4^m-1)+1, \frac{4}{3}(4^m-1)+2\}\;\text{ for some }\; m\geq 1 ; \\
                       0,  & \text{ otherwise }  .
                     \end{array}
\end{cases}
\end{equation*}

\item We have
\begin{equation*}
H_{n,n}=\begin{cases}
                       1, & n=\sum_{i=1}^{m}4^{k_{i}},\quad\text{where}\quad k_{i}<k_{i+1}; \\
                       0, & \text{ otherwise }.
\end{cases}
\end{equation*}
\end{enumerate}
\end{conj}

\begin{rem}
{\rm Let us observe that $H_{n,n}=1$ if and only if $n$ is an element of the Moser-de Bruijn sequence \cite{Bru, Mos}.}
\end{rem}

\begin{prob}
Let $m$ be a non-zero integer and let $F_{1}(X)$ ($F_{-1}(X)$ respectively) be generating function for the PTM sequence (for the iPTM sequence respectively) treated as an element of $\mathbb{F}_{2}[[X]]$. Next, let us define
$$
F_{m}(X)=\begin{cases}\begin{array}{cl}
                        F_{1}^{(m)}(X)   & \mbox{for}\quad m\geq 1 \\
                        F_{-1}^{(-m)}(X) & \mbox{for}\quad m\leq -1
                      \end{array}
\end{cases},
$$
where as usual $F^{(m)}(X)=F\circ F\circ \ldots \circ F$ is composition of $m$ copies of $F$. We write
$$
F_{m}(X)=\sum_{n=0}^{\infty}C_{m,n}X^{n}
$$
and ask about arithmetic properties of the sequence ${\bf C}_{m}=(C_{m,n})_{n\in\N}$.
\end{prob}

Finally, in the light of results obtained in this paper it is natural to state the following:

\begin{prob}
Let $p$ be a prime number and consider the series $F_{p}(X)=\sum_{n=0}^{\infty}s_{p}(n)\pmod{p}X^{n}\in\mathbb{F}_{p}[[X]]$. Study the sequence ${\bf c}_{p}=(c_{p,n})_{n\in \N}$, where $c_{p,n}$ is $n$-th coefficient in expansion of the power series $G_{p}(X)$ satisfying the identity $G_{p}(F_{p}(X))=F_{p}(G_{p}(X))=X$.

In particular, let $u\in\{0,1,\ldots,p-1\}$ and put
\begin{equation*}
\cal{A}_{p,u}:=\{n:\;c_{p,n}=u\}.
\end{equation*}
Characterize these $p$'s and $u\in\{0,1,\ldots,p-1\}$ for which the increasing sequence of elements of the set $\cal{A}_{u,p}$ is $k$-regular for some $k$.
\end{prob}

\bigskip

\noindent {\bf Acknowledgments}
We express our gratitude to the anonymous referee for a careful reading of the manuscript and many valuable suggestions and corrections made.

 \bigskip

\noindent  Maciej Gawron and Maciej Ulas, Jagiellonian University, Faculty of Mathematics and Computer Science, Institute of Mathematics, {\L}ojasiewicza 6, 30 - 348 Krak\'{o}w, Poland\\
e-mail:\;{\tt \{maciej.ulas, maciej.gawron\}@uj.edu.pl}

 \end{document}